\documentclass[11pt,reqno]{article}
\usepackage{amssymb}
\usepackage{amsmath}
\usepackage{amsthm}
\usepackage{amsfonts}
\usepackage{url}
\usepackage{hyperref}
\usepackage{breakurl}

\DeclareMathOperator{\dup}{d\hspace{-1.5pt}}

\newcommand{\raisecomma}{\raisebox{2pt}{$,$}}
\newcommand{\raisedot}{\raisebox{2pt}{$.$}}

\newcommand{\rhat}{{\widehat{r}}}
\newcommand{\rtilde}{{\widetilde{r}}}
\newcommand{\Jtilde}{{\widetilde{J}}}
\newcommand{\Gammatilde}{{\widetilde{\Gamma}}}
\newcommand{\thetahat}{{\widehat{\theta}}}
\newcommand{\thetatilde}{{\widetilde{\theta}}}
\newcommand{\betahat}{{\widehat{\beta}}}
\newcommand{\betatilde}{{\widetilde{\beta}}}
\newcommand{\half}{{\textstyle\frac12}}

\newcommand{\C}{{\mathbb C}}

\newtheorem{corollary}{Corollary}
\newtheorem{lemma}{Lemma}

\newtheorem{remark}{Remark}

\begin{document}
\bibliographystyle{plain}
\title{~\\[-60pt]
Asymptotic approximation of central binomial coefficients with
rigorous error bounds
}
\author{Richard P.\ Brent\\[5pt]
Australian National University\\
Canberra, ACT 2600,
Australia
}

\date{}

\maketitle
\thispagestyle{empty}                   

\begin{abstract}
We show that a well-known asymptotic series for the logarithm of the central
binomial coefficient is strictly enveloping
in the sense of P\'olya and Szeg\"o,
so the error incurred in truncating the series is of the same sign as the
next term, and is bounded in magnitude by that term. We consider
closely related asymptotic series for Binet's function,
for $\ln\Gamma(z+\frac12)$,
and for the Riemann-Siegel theta function,
and make some historical remarks.\\


\end{abstract}
 
\pagebreak[3]

\section{Introduction}		\label{sec:intro}

Let $z\in\C$ and assume that $\Re z > 0$.
It is well-known that
\begin{equation}		\label{eq:Stirling}
\ln\Gamma(z) = (z-\half)\ln z - z + \half\ln(2\pi) + J(z),
\end{equation}
where $J(z)$ can be written as
\begin{equation}		\label{eq:Binet}
J(z) = \frac{1}{\pi}\int_0^\infty\frac{z}{\eta^2+z^2}
	\ln\left(\frac{1}{1-e^{-2\pi\eta}}\right)\,\dup\eta\,.
\end{equation}
The analytic function $J(z)$ is known as
\emph{Binet's function} and has several equivalent
expressions; see for example Henrici~\cite[(8.5-7)]{Henrici-v2}.

Binet's function has an asymptotic expansion
\begin{equation}		\label{eq:Binet-expansion}
J(z) \sim \frac{\beta_0}{z} - \frac{\beta_1}{z^3}
   + \frac{\beta_2}{z^5} - \cdots,
\end{equation}
or more precisely, for non-negative integers $k$,
\begin{equation}		\label{eq:Binet-expansion-remainder}
J(z) = \sum_{j=0}^{k-1} (-1)^j\frac{\beta_j}{z^{2j+1}} + r_k(z),
\end{equation}
where
\begin{equation}		\label{eq:beta}
\beta_k = \frac{1}{\pi}\int_0^\infty
        \eta^{2k}
        \ln\left(\frac{1}{1-e^{-2\pi\eta}}\right)\,\dup\eta
\end{equation}
and
\begin{equation}		\label{eq:Binet_rk}
r_k(z) = \frac{(-1)^k}{\pi z^{2k-1}}\int_0^\infty
	\frac{\eta^{2k}}{z^2+\eta^2}
	\ln\left(\frac{1}{1-e^{-2\pi\eta}}\right)\,\dup\eta\,.
\end{equation}
It may be shown that
\begin{equation}	\label{eq:betakzetaB}
\beta_k = \frac{2(2k)!}{(2\pi)^{2k+2}}\,\zeta(2k+2)
	= \frac{(-1)^k}{(2k+1)(2k+2)}\,B_{2k+2}\,,
\end{equation}
where $B_{2k+2}$ is a Bernoulli number ($B_2 = 1/6, B_4 = -1/30$, etc.).
Proofs of these results are given in
Henrici's book~\cite[\S11.1]{Henrici-v2}.\footnote{There
is an error in Henrici's equation (11.1-13):
$2^{-2\pi\eta}$ should be replaced by $e^{-2\pi\eta}$.}
As far as possible, we have followed Henrici's notation.

Substituting \eqref{eq:Binet-expansion-remainder} into
\eqref{eq:Stirling} gives an asymptotic expansion for $\ln\Gamma(z)$ that
is usually named after James Stirling,
although some credit is due to Abraham de Moivre.
For the history and early references, see Dutka~\cite{Dutka}.
It is interesting to note that de Moivre started (about 1721) 
by trying to approximate
the central binomial coefficient $\binom{2n}{n}$, not the factorial (or
Gamma) function~-- see Dutka~\cite[pg.~227]{Dutka}.

It is easy to see from~\eqref{eq:beta} and~\eqref{eq:Binet_rk} that
\begin{equation}		\label{eq:r_k_theta}
r_k(z) = \theta_k(z)\, (-1)^k\frac{\beta_k}{z^{2k+1}}\,\raisecomma
\end{equation}
where
\begin{equation}		\label{eq:theta}
\theta_k(z) = \int_0^\infty
	\frac{z^2\,\eta^{2k}}{z^2+\eta^2}
	\ln\left(\frac{1}{1-e^{-2\pi\eta}}\right)\,\dup\eta\
	\Bigg/ 
\int_0^\infty
        \eta^{2k}
        \ln\left(\frac{1}{1-e^{-2\pi\eta}}\right)\,\dup\eta\,.
\end{equation}
Suppose now that $z$ is real and positive.
Since $z^2/(z^2+\eta^2) \in (0,1)$ and the logarithmic factors
in~\eqref{eq:theta} are positive for all $\eta\in(0,\infty)$,
we see that
\begin{equation}			\label{eq:enveloping}
	\theta_k(z) \in (0,1).
\end{equation}
Thus, the remainder $r_k(z)$ given by~\eqref{eq:r_k_theta} has the same
sign as the next term $(-1)^k\beta_k/z^{2k+1}$
in the asymptotic series, and is smaller
in absolute value. In the terminology used
by P\'olya and Szeg\"o~\cite[Ch.~4]{PS-v1},
the asymptotic series
for $\ln\Gamma(z)$ \emph{strictly envelops}\footnote{We refer to the English
translation. In the original it is ``in engerem Sinne umh\"ullen''.} 
the function $\ln\Gamma(z)$.%
\footnote{When testing the enveloping property,
we only consider the nonzero terms in the asymptotic series.
See~\cite[Problem 142, footnote 1]{PS-v1}.}

Section \ref{sec:central} shows that we can deduce a strictly enveloping
asymptotic series for $\ln(\Gamma(2z+1)/\Gamma(z+1)^2)$ or equivalently, 
if $z=n$ is
a positive integer, for the logarithm of
the central binomial coefficient $\binom{2n}{n}$.
The series itself is well known,
but we have not found the
enveloping property or the resulting error bound
mentioned in the literature.  Henrici
was aware of it, since in his book~\cite[\S11.2, Problem~6]{Henrici-v2}
he gives the special case $k=3$ as an exercise, along with a hint
for the solution. Hence, we do not claim any particular originality.
Our purpose is primarily to make some useful asymptotic
series and their associated error bounds readily accessible.
Related results and additional references may be found,
for example, in~\cite{AR,Nemes13,Olver}.

In \S\ref{sec:central} we consider the central binomial coefficient and its
generalisation to a complex argument.  Then, in \S\ref{sec:other}, we
consider some closely related asymptotic series that we can prove to be
strictly enveloping. In~\S\ref{sec:non-env} we make some remarks on
asymptotic series that are \emph{not} enveloping.
An Appendix gives numerical values of the coefficients appearing
in three of the asymptotic series.

Finally,
we remark that it is possible to give asymptotic 
series related to $\Gamma(z+\frac12)/\Gamma(z)$
and $\binom{2n}{n}$, but in general these
series do not alternate in sign. 
See, for example,
\cite{Dyson}, \cite{Elezovic}, 
\cite[ex.~9.60 and pg.~602]{Graham},
\cite{KS}, 
and \cite{Lehmer}. 

\pagebreak[3]
\section{Asymptotic series for central binomial coefficients}
	\label{sec:central}

Define
\begin{align*}
\Gammatilde(z) :=& \frac{\Gamma(2z+1)}{\Gamma(z+1)^2}\,\raisecomma\\
\Jtilde(z) :=& J(2z) - 2J(z),\\
\rtilde_k(z) :=& r_k(2z) - 2r_k(z),
\end{align*}
and
\begin{equation}	\label{eq:betatilde}
\betatilde_k := (2-2^{-2k-1})\beta_k = 
  (-1)^{k}\frac{(1-4^{-k-1})}{(k+1)(2k+1)}B_{2k+2}\,.
\end{equation}
As noted above,
the central binomial coefficient $\binom{2n}{n}$ is simply
$\Gammatilde(n)$.

Using elementary properties of the Gamma function, we have
\begin{equation}			\label{eq:Gammatilde}
\Gammatilde(z) = \frac{2}{z} \frac{\Gamma(2z)}{\Gamma(z)^2}\,\raisedot
\end{equation}
Thus, from~\eqref{eq:Stirling} and the same equation with $z \mapsto 2z$,
we have
\begin{equation}			\label{eq:lnGammatildeJtilde}
\ln\Gammatilde(z) = \ln\left(\frac{4^z}{\sqrt{\pi z}}\right) + \Jtilde(z).
\end{equation}
Also, from~\eqref{eq:Binet-expansion-remainder}
and the definition of $\Jtilde(z)$, we have an asymptotic series
for $\Jtilde(z)$, namely:
\begin{equation}	\label{eq:Jtilde-series}
\Jtilde(z) = -\sum_{j=0}^{k-1} (-1)^j\frac{\betatilde_j}{z^{2j+1}} + 
	     \rtilde_k(z).
\end{equation}
Since $\binom{2n}{n} = \Gammatilde(n)$,
equations~\eqref{eq:lnGammatildeJtilde}--\eqref{eq:Jtilde-series} 
give an asymptotic series for $\ln\binom{2n}{n}$.
Lemma~\ref{lem:lemma1} shows that the
remainder $\rtilde_k(z)$ can be expressed as an integral analogous to
the integral~\eqref{eq:Binet_rk} for $r_k(z)$.
\begin{lemma}				\label{lem:lemma1}
For $z\in\C, \Re z > 0$, and $k$ a non-negative integer,
\begin{equation}			\label{eq:lemma1a}
\betatilde_k = - \frac{1}{\pi}\int_0^\infty
	\eta^{2k}
	\ln\tanh(\pi\eta)\dup\eta\,,
\end{equation}
\begin{equation}			\label{eq:lemma1b}
\rtilde_k(z) = \frac{(-1)^k}{\pi z^{2k-1}}\int_0^\infty
	\frac{\eta^{2k}}{z^2+\eta^2}
	\ln\tanh(\pi\eta)\dup\eta\,,
\end{equation}
and
\begin{equation}			\label{eq:Jr}
	\Jtilde(z) = \rtilde_0(z).
\end{equation}
\end{lemma}
\begin{proof}
Making the change of variables $z \mapsto 2z$ and
$\eta \mapsto 2\eta$ in~\eqref{eq:Binet_rk}, we obtain
\[
r_k(2z) = \frac{(-1)^k}{\pi z^{2k-1}}\int_0^\infty
        \frac{\eta^{2k}}{z^2+\eta^2}
        \ln\left(\frac{1}{1-e^{-4\pi\eta}}\right)\,\dup\eta\,.
\]
Now
\begin{equation*}
\ln\left(\frac{1}{1-e^{-4\pi\eta}}\right) -
2\ln\left(\frac{1}{1-e^{-2\pi\eta}}\right)
= \ln\left(\frac{1-e^{-2\pi\eta}}{1+e^{-2\pi\eta}}\right)
= \ln\tanh(\pi\eta),
\end{equation*}
so~\eqref{eq:lemma1b}--\eqref{eq:Jr} follow from the
definitions of $\rtilde_k(z)$ and $\Jtilde(z)$.
The proof of~\eqref{eq:lemma1a} is similar.
\end{proof}
\pagebreak[3]

Corollary~\ref{cor:cor1} gives a result analogous
to equations~\eqref{eq:r_k_theta}--\eqref{eq:theta}.
\begin{corollary}			\label{cor:cor1}
For $z\in\C, \Re z > 0$, and $k$ a non-negative integer,
\begin{equation}			\label{eq:rtildek}
\rtilde_k(z) = \thetatilde_k(z)(-1)^{k+1}\frac{\betatilde_k}{z^{2k+1}}
	\,\raisecomma
\end{equation}
where
\begin{equation}			\label{eq:thetatilde}
\thetatilde_k(z) = \int_0^\infty \frac{z^2\,\eta^{2k}}{z^2+\eta^2}
	\ln\tanh(\pi\eta)\dup\eta
	\Bigg/
	\int_0^\infty \eta^{2k}
        \ln\tanh(\pi\eta)\dup\eta\,.
\end{equation}
\end{corollary}
\begin{proof}
This is straightforward from equations~\eqref{eq:lemma1a}--\eqref{eq:lemma1b}
of Lemma~\ref{lem:lemma1}.
\end{proof}
Corollary~\ref{cor:cor2} gives a result analogous to
the bound~\eqref{eq:enveloping}.
\begin{corollary}		\label{cor:cor2}
If $z$ is real and positive, then $\thetatilde_k(z) \in (0,1)$.
\end{corollary}
\begin{proof}
We write~\eqref{eq:thetatilde} as
\begin{equation}                        \label{eq:thetatilde2}
\thetatilde_k(z) = \frac{\displaystyle
	\int_0^\infty \frac{z^2\,\eta^{2k}}{z^2+\eta^2}
        \ln\coth(\pi\eta)\dup\eta}
        {\displaystyle
        \int_0^\infty \eta^{2k}
        \ln\coth(\pi\eta)\dup\eta}\,\raisedot
\end{equation}
Observe that
$\coth(y) = \cosh(y)/\sinh(y) > 1$ for $y\in(0,\infty)$,
so $\ln\coth(y) > 0$ for $y = {\pi\eta} > 0$.
Since $z^2/(z^2+\eta^2) \in (0,1)$ for real positive $z$ and $\eta$,
the result follows.
\end{proof}
\begin{corollary}			\label{cor:cor3}
If $z$ is real and positive, then the asymptotic
series~\eqref{eq:Jtilde-series} for $\Jtilde(z)$ is strictly enveloping.
\end{corollary}
\begin{proof}
This is immediate from Corollary~\ref{cor:cor2}.
\end{proof}
\begin{remark}
{\rm 
We may compare
Corollary~\ref{cor:cor2} with (the proof of)
Lemma~2.7 of~\cite{rpb262}.
The latter, after allowing for different notation, 
gives the bound
\[
\frac{-1}{4^{k+1}-1} < \thetatilde_k(z) < \frac{4^{k+1}}{4^{k+1}-1}.
\]
This is clearly weaker than the bound of Corollary~\ref{cor:cor2},
and not sufficient to prove Corollary~\ref{cor:cor3}.
} 
\end{remark}

\pagebreak[3]
\section{Some related asymptotic series}	\label{sec:other}

\begin{lemma}				\label{lem:Gamma_ratio}
If $z\in\C, \Re z > 0$, then
\[
\Jtilde(z) = \ln\left(\frac{\Gamma(z+\frac12)}{z^{1/2}\Gamma(z)}\right)
	\raisedot
\]
\end{lemma}
\begin{proof}
This follows from equations~\eqref{eq:Gammatilde}--\eqref{eq:lnGammatildeJtilde}
and the duplication formula 
$\Gamma(z)\Gamma(z+\half) = 2^{1-2z}\pi^{1/2}\Gamma(2z)$.
\end{proof}
{From} Lemma~\ref{lem:Gamma_ratio} and~\eqref{eq:Jtilde-series} we immediately
obtain an asymptotic expansion
\begin{equation}		\label{eq:logGrat}
\ln\left(\frac{\Gamma(z+\frac12)}{\Gamma(z)}\right)
	\sim \frac{\ln z}{2} + \sum_{j \ge 0}(-1)^{j+1}
	\frac{\betatilde_j}{z^{2j+1}}
\end{equation}
which is strictly enveloping if $z$ is real and positive.

Define 
\begin{equation}		\label{eq:betahat}
\betahat_j = \betatilde_j - \beta_j = (1-2^{-2j-1})\beta_j.
\end{equation}
Using the asymptotic expansion for $\ln\Gamma(z)$ given by equations
\eqref{eq:Stirling} and \eqref{eq:Binet-expansion}, we see 
from~\eqref{eq:logGrat} that $\ln\Gamma(z+\frac12)$ has an asymptotic
expansion
\begin{equation}		\label{eq:lnGammahalf}
\ln\Gamma(z+\half) \sim z\ln z - z + 
	\half\ln(2\pi)
	+ \sum_{j \ge 0}(-1)^{j+1}\,\frac{\betahat_j}{z^{2j+1}}
	\,\raisedot
\end{equation}
In fact, the expansion~\eqref{eq:lnGammahalf} was already obtained
by Gauss~\cite[Eqn.~{[}59{]} of
Art.~29]{Gauss-v3} 
in 1812.
However, Gauss did not explicitly bound the truncation error.
Writing~\eqref{eq:lnGammahalf} as
\begin{equation}		\label{eq:lnGammahalf2}
\ln\Gamma(z+\half) = z\ln z - z + 
	\half\ln(2\pi)
	+ \sum_{j=0}^{k-1}(-1)^{j+1}\,\frac{\betahat_j}{z^{2j+1}}
	+ \rhat_k(z)
	\,\raisecomma
\end{equation}
we have an unsurprising result for the truncation error $\rhat_k(z)$:
the error is of the same sign as the first neglected term
$(-1)^{k+1}\betahat_k/z^{2k+1}$,
and is bounded in magnitude by this term. This is shown in
Lemma~\ref{lem:lemma3} and Corollaries~\ref{cor:cor4}--\ref{cor:cor5}
below.

\begin{lemma}				\label{lem:lemma3}
For $z\in\C, \Re z > 0$, and $k$ a non-negative integer,
\begin{equation}			\label{eq:lemma3a}
\betahat_k = \frac{1}{\pi}\int_0^\infty
	\eta^{2k}
	\ln\left(1+e^{-2\pi\eta}\right)\dup\eta
\end{equation}
and
\begin{equation}			\label{eq:lemma3b}
\rhat_k(z) = \frac{(-1)^{k+1}}{\pi z^{2k-1}}\int_0^\infty
	\frac{\eta^{2k}}{z^2+\eta^2}
	\ln\left(1+e^{-2\pi\eta}\right)\dup\eta\,.
\end{equation}
\end{lemma}
\begin{proof}
This is similar to the proof of Lemma~\ref{lem:lemma1}.
\end{proof}
\begin{corollary}			\label{cor:cor4}
For $z\in\C, \Re z > 0$, and $k$ a non-negative integer,
\begin{equation}			\label{eq:rhatk}
\rhat_k(z) = \thetahat_k(z)(-1)^{k+1}\frac{\betahat_k}{z^{2k+1}}\,\raisecomma
\end{equation}
where
\begin{equation}			\label{eq:thetahat}
\thetahat_k(z) = \int_0^\infty \frac{z^2\,\eta^{2k}}{z^2+\eta^2}
	\ln\left(1+e^{-2\pi\eta}\right)\dup\eta
	\Bigg/
	\int_0^\infty \eta^{2k}
        \ln\left(1+e^{-2\pi\eta}\right)\dup\eta\,.
\end{equation}
\end{corollary}
\begin{proof}
This is a straightforward consequence of Lemma~\ref{lem:lemma3}.
\end{proof}
\begin{corollary}			\label{cor:cor5}
If $z$ is real and positive, then the asymptotic expansion
for $\ln\Gamma(z+\half)$ given in~\eqref{eq:lnGammahalf2} is
strictly enveloping.
\end{corollary}
\begin{proof}
{From}~\eqref{eq:thetahat} we have $\thetahat_k(z) \in (0,1)$.
\end{proof}

\pagebreak[3]
\begin{remark}
{\rm
If we make the change of variables $z \mapsto n+\frac12$
in~\eqref{eq:lnGammahalf}, and assume that $n$ is a positive integer,
we obtain an asymptotic series for $n!$ in negative powers of $(n+\half)$:
\begin{equation}		\label{eq:deMoivre}
\ln n! \sim (n+\half)\ln(n+\half) - (n+\half) + \half\ln(2\pi)
	+ \sum_{j\ge 0}(-1)^{j+1}\,\frac{\betahat_j}{(n+\half)^{2j+1}}
	\,\raisedot
\end{equation}
In fact, \eqref{eq:deMoivre} was stated (without proof) by 
de Moivre~\cite{deMoivre2,deMoivre3} as early as 1730, 
see Dutka~\cite[(5), pg.~233]{Dutka}.
} 
\end{remark}

\section{Non-enveloping asymptotic series} \label{sec:non-env}

Lest the reader has gained the impression that all ``naturally occurring''
asymptotic series are enveloping (for real positive arguments), 
we give two classes of examples to show that this
is not the case. In fact, enveloping series are the exception, not the
rule. 
Our first class of examples is given by the following Lemma.
\begin{lemma}		\label{lem:fJ}
Let $x\in(0,+\infty)$ and  $f(x) := J(x) + \exp(-bx)$ 
for some constant $b \in (0, 2\pi)$.
Then $f(x)$ has an asymptotic series
\begin{equation}		\label{eq:f-asym}
f(x) \sim \sum_{j=0}^{\infty} (-1)^j\frac{\beta_j}{x^{2j+1}}\,\raisedot
\end{equation}
However, the series~\eqref{eq:f-asym} does not envelop $f$.
\end{lemma}
\begin{proof}
For all $k \ge 0$,
$\exp(-bx) = O(x^{-2k-1})$ as $x \to +\infty$.
Thus, it follows from~\eqref{eq:Binet-expansion-remainder} that
$f(x)$ has the claimed asymptotic series (in fact the same series
as the Binet function $J$.)
This proves the first claim.

To prove the final claim,
suppose, by way of contradiction, that
the series~\eqref{eq:f-asym} envelops $f$.
For each integer $k > 0$,
define $x_k := k/\pi$.
{From} \eqref{eq:betakzetaB}, the $\beta_k$ grow like
$(2k)!/(2\pi)^{2k}$, 
and from Stirling's approximation we see that
\begin{equation}		\label{eq:constraint}
\beta_k/x_k^{2k+1} = O(\exp(-2\pi x_k))
 \;\text{ as }\; k \to \infty. 
\end{equation}
Since the same series envelops both $f$ and $J$,
\eqref{eq:constraint} implies that
\[
|f(x_k)-J(x_k)| = O(\exp(-2\pi x_k))
 \;\text{ as }\; k \to \infty.
\]
Since $\exp(-2\pi x) = o(\exp(-bx))$, it follows that,
for sufficiently large $k$,
\[
|f(x_k)-J(x_k)| < \exp(-bx_k).
\]
This contradicts the definition of $f$,
so the assumption that the series~\eqref{eq:f-asym} envelops $f$
must be false.
\end{proof}
\begin{remark}
{\rm
Lemma~\ref{lem:fJ} can be generalised.  For example, the
conclusion holds if $f(x) = J(x) + g(x)$,
where $g(x) = O(x^{-k})$ for all positive integers $k$, 
but $g(x) \ne O(\exp(-2\pi x))$.
Also, we can replace the function $J(x)$ by a different function
that has an enveloping asymptotic series whose terms grow at
the same rate as those of $J(x)$.
} 
\end{remark}

Our second class of examples involves asymptotic expansions where all 
(or all but a finite number) of the terms are of the same sign
(assuming a positive real argument $x$).
Such series can not be strictly enveloping~\cite[Ch.~4]{PS-v1}.
As examples, we mention the Bessel function $I_0(x)$
(see Olver and Maximon~\cite[\S10.40.1]{OM}),
the product of two Bessel functions $I_0(x)K_0(x)$
(see \cite[\S10.40.6]{OM} and \hbox{\cite[Lemma~3.1]{rpb256}}),
and the {Riemann-Siegel theta function} (see~\cite[\S6.5]{Edwards}).
In all these examples the terms have constant sign, so the remainder changes
monotonically as the number of terms increases with the argument $x$ fixed.
Eventually the remainder changes sign and starts increasing in absolute
value. 
Often the point where the remainder changes sign
is close to where the terms are smallest in absolute value,
but this is not always true~--
see for example~\cite[\S\S4--5]{rpb268}.

\pagebreak[3]
\section{Concluding remarks}	\label{sec:conc}

We have considered three different but related asymptotic series
that can all be proved to be strictly enveloping.  Our proofs depend on the
fact that the three relevant functions
$-\ln(1-e^{-2\pi\eta})$, $\ln\coth(\pi\eta)$,
and $\ln(1+e^{-2\pi\eta})$ are positive for all $\eta\in(0,\infty)$.
We remark that these three functions are linearly dependent,
since
\[
\coth(\pi\eta) = \frac{1+e^{-2\pi\eta}}{1-e^{-2\pi\eta}}\,\raisedot
\]
It follows that the sequences $(\beta_k)_{k\ge 0}$, 
$(\betatilde_k)_{k\ge 0}$ and $(\betahat_k)_{k\ge 0}$
are linearly dependent. In fact,
$\betatilde_k = \beta_k + \betahat_k$ for all $k\ge 0$.
A table of numerical values is given in the Appendix.\\

\noindent\textbf{Acknowledgement:}
We thank an anonymous referee
for simplifying the proof of Lemma~\ref{lem:fJ} and 
for noting that the domain of validity of~\eqref{eq:Binet}
is the right half-plane $\Re z > 0$
(although the Binet function $J(x)$ 
may be continued into the left half-plane by 
analytic continuation, see~\cite[Thm.~8.5a]{Henrici-v2}).\\

\noindent\textbf{Conflicts of Interest:} The author declares 
no conflict of interest.

\pagebreak[3]

\pagebreak[3]
\section*{Appendix: numerical values of the coefficients}

The table below gives the exact values of the coefficients $\beta_k$,
$\betatilde_k$ and $\betahat_k$ for $0 \le k \le 6$.
The values have been computed from equations 
\eqref{eq:betakzetaB}, \eqref{eq:betatilde} and \eqref{eq:betahat}.
We recall from the discussion above that the 
coefficients occur in the asymptotic expansions
\begin{align*}	
\ln\Gamma(z) \sim&\; (z-\half)\ln z - z + \half\ln(2\pi)
	+ \frac{\beta_0}{z} - \frac{\beta_1}{z^3}
	+ \frac{\beta_2}{z^5} - \cdots,\\
\ln\binom{2n}{n} \sim&\; \ln\left(\frac{4^n}{\sqrt{\pi n}}\right)
	- \frac{\betatilde_0}{n} + \frac{\betatilde_1}{n^3}
	- \frac{\betatilde_2}{n^5} + \cdots,\;\text{ and}\\
\ln\Gamma(z+\half) \sim&\; z\ln z - z + \half\ln(2\pi)
        - \frac{\betahat_0}{z} + \frac{\betahat_1}{z^3}
        - \frac{\betahat_2}{z^5} + \cdots,
\end{align*}
the $\betahat_k$ also occurring in de Moivre's
series~\eqref{eq:deMoivre} and, with a different sign pattern,
in a series
related to the Riemann-Siegel theta function
\cite[eqn.~(2)]{rpb268}:
$2\vartheta(t) \sim t\ln(t/2\pi e) - \pi/4 + \betahat_0/t
 + \betahat_1/t^3 + \cdots$.
In all but the Riemann-Siegel theta function case 
the asymptotic series are strictly enveloping, so the error
incurred in truncating the series can be bounded by the first
term omitted, provided
that $z\in(0,\infty)$ is real and that $n$ is a positive integer.
For error bounds if $z$ is complex,
we refer to~\cite[\S\S2--3]{rpb268}.

\begin{table}[ht]
\begin{center}
\begin{tabular}{cccc}
$k$ & $\beta_k$ & $\betatilde_k$ & $\betahat_k$\\
\hline
&&\\[-9pt]
$0$& $1/12$ & $1/8$ & $1/24$\\[1pt]
$1$& $1/360$ & $1/192$ & $7/2880$\\[1pt]
$2$& $1/1260$& $1/640$ & $31/40320$\\[1pt]
$3$& $1/1680$& $17/14336$ & $127/215040$\\[1pt]
$4$& $1/1188$& $31/18432$ & $511/608256$\\[1pt]
$5$& $691/360360$& $691/180224$ & $1414477/738017280$\\[1pt]
$6$& $1/156$& $5461/425984$ & $8191/1277952$\\[1pt]
\hline
\end{tabular}
\end{center}
\end{table}

\begin{remark}
{\rm
We note that the sequence $((-1)^k\beta_k)_{k\ge 0}$ is in the Online
Encyclopedia of Integer Sequences (OEIS)~\cite{OEIS}.  
The (signed) numerators are 
sequence A046968, and the denominators are sequence A046969.  The sequence
$(\betahat_k/2)_{k\ge 0}$ is also in the OEIS: the numerators are
sequence A036282, and the denominators are sequence A114721.
We have added the sequence $((-1)^k\betatilde_k)_{k\ge 0}$ to the OEIS.
The (signed) numerators are sequence A275994, and the denominators
are sequence A275995.  
Other relevant sequences are A143503
and A061549.
} 
\end{remark}

\end{document}